\documentclass[11pt]{article}
\usepackage{amsmath, amsfonts, amsthm, amssymb}
\usepackage{graphicx}
\usepackage{float}
\usepackage{verbatim}% for comment

\numberwithin{equation}{section}
\newtheorem{thm}{Theorem}[section]
\newtheorem{lma}[thm]{Lemma}
\newtheorem{cor}[thm]{Corollary}

\newtheorem{prop}[thm]{Proposition}

\renewcommand{\geq}{\geqslant}
\renewcommand{\leq}{\leqslant}

\allowdisplaybreaks

\title{Inhomogeneous self-affine carpets}

\author{Jonathan M. Fraser\\ \\
\emph{Mathematics Institute, Zeeman Building,}\\ \emph{University of Warwick, Coventry CV4 7AL, UK}\\ \emph{e-mail: jon.fraser32@gmail.com}}

\begin{document}
\maketitle

\begin{abstract}
We investigate the dimension theory of inhomogeneous self-affine carpets.  Through the work of Olsen, Snigireva and Fraser, the dimension theory of inhomogeneous \emph{self-similar} sets is now relatively well-understood, however, almost no progress has been made concerning more general non-conformal inhomogeneous attractors.  If a dimension is countably stable, then the results are immediate and so we focus on the upper and lower box dimensions and compute these explicitly for large classes of inhomogeneous self-affine carpets.  Interestingly, we find that the `expected formula' for the upper box dimension can fail in the self-affine setting and we thus reveal new phenomena, not occurring in the simpler self-similar case.
\\ \\
\emph{Mathematics Subject Classification} 2010:  primary: 28A80, 26A18. \\ \\
\emph{Key words and phrases}: inhomogeneous attractor, self-affine carpet, box dimension.
\end{abstract}

\section{Introduction}

Inhomogeneous self-similar sets were introduced by Barnsley and Demko \cite{barndemko} as a natural generalisation of the standard iterated function system construction. As well as being interesting from a structural and geometric point of view, these sets have useful applications in image compression.  The dimension theory of inhomogeneous self-similar sets was investigated by Olsen and Snigireva \cite{olseninhom, ninaphd} and later by Fraser \cite{me_inhom}, revealing many subtle differences between the inhomogeneous attractors and their homogeneous counterparts.  Up until now, this investigation has only been concerned with self-similar constructions, where the arguments are considerably simplified because the scaling properties are very well understood.  In particular, as one zooms into the set, one sees the same picture as one started with, even in the inhomogeneous case.  In this paper we take on the more challenging problem of studying inhomogeneous self-affine sets, and, in particular, inhomogeneous versions of the self-affine carpets considered by Bedford-McMullen \cite{bedford, mcmullen}, Lalley-Gatzouras \cite{lalley-gatz} and Bara\'nski \cite{baranski}.

\subsection{Inhomogeneous attractors: structure and dimension}

Inhomogeneous iterated function systems are generalisations of standard iterated function systems (IFSs), which are one of the most important methods for constructing fractal sets.  Indeed, one might call the attractors of the standard systems \emph{homogeneous attractors}.  Let $(X,d)$ be a compact metric space.  An IFS is a finite collection $\{ S_i \}_{i \in \mathcal{I}}$ of contracting self-maps on $X$.  It is a fundamental result in fractal geometry, dating back to Hutchinson's seminal 1981 paper \cite{hutchinson}, that for every IFS there exists a unique non-empty compact set $F$, called the \emph{attractor}, which satisfies
\[
F = \bigcup_{i\in \mathcal{I}} S_i (F).
\]
This can be proved by an elegant application of Banach's contraction mapping theorem.  If an IFS consists solely of \emph{similarity} transformations, then the attractor is called a \emph{self-similar set}.  Likewise, if $X$ is a Euclidean space and the mappings are all translate linear (\emph{affine}) transformations, then the attractor is called \emph{self-affine}.
\\ \\
More generally, given a compact set $C \subseteq X$, called the \emph{condensation set},  analogous to the homogeneous case, there is a unique non-empty compact set $F_C$ satisfying
\begin{equation} \label{inhom}
F_C =  \bigcup_{i \in \mathcal{I}} S_i (F_C) \ \cup \ C,
\end{equation}
which we refer to as the \emph{inhomogeneous} attractor (with condensation $C$).  Note that homogeneous attractors are inhomogeneous attractors with condensation equal to the empty set.  From now on we will assume that the condensation set is non-empty.  Inhomogeneous attractors were introduced and studied in \cite{barndemko} and are also discussed in detail in \cite{superfractals} where, among other things, Barnsley gives applications of these schemes to image compression, see Figure 1.  Inhomogeneous attractors also have interesting applications in the realm of dimension theory and fractals: in certain cases homogeneous attractors with complicated overlaps can be viewed as inhomogeneous attractors without overlaps, making them considerably easier to study.  See for example \cite[Section 2.1]{olseninhom2}.  Let $\mathcal{I}^* = \bigcup_{k\geq1} \mathcal{I}^k$ denote the set of all finite sequences with entries in $\mathcal{I}$ and for
\[
\textbf{\emph{i}}= \big(i_1, i_2, \dots, i_k \big) \in \mathcal{I}^*
\]
write
\[
S_{\textbf{\emph{i}}} = S_{i_1} \circ S_{i_2} \circ \dots \circ S_{i_k}.
\]
The \emph{orbital set},
\[
\mathcal{O} \ = \ C \ \cup \  \bigcup_{\textbf{\emph{i}} \in \mathcal{I}^*} S_{\textbf{\emph{i}}} (C),
\]
 was introduced in \cite{superfractals} and it turns out that this set plays an important role in the structure of inhomogeneous attractors.  In particular,
\begin{equation} \label{structure}
F_C \ = \  F_\emptyset \cup \mathcal{O} \ = \  \overline{\mathcal{O}},
\end{equation}
where $F_\emptyset$ is the \emph{homogeneous} attractor of the IFS.  The relationship (\ref{structure}) was proved in \cite[Lemma 3.9]{ninaphd} in the case where $X$ is a compact subset of $\mathbb{R}^d$ and the maps are similarities.  We note here that their arguments easily generalise to obtain the general case stated above.  When considering the dimension $\dim$ of $F_C$, one expects the relationship
\begin{equation} \label{inhomexpected}
\dim F_C = \max \{ \dim  F_\emptyset, \ \dim C\}
\end{equation}
to hold.  Indeed, if $\dim$ is countably stable, monotone and does not increase under Lipschitz maps, then
\newpage
\begin{eqnarray*}
\max \{ \dim  F_\emptyset, \ \dim C\} \ \leq \ \dim F_C &=& \dim (F_\emptyset \cup \mathcal{O}) \\ \\
&=& \max \Big\{ \dim  F_\emptyset, \ C \ \cup \  \bigcup_{\textbf{\emph{i}} \in \mathcal{I}^*} S_{\textbf{\emph{i}}} (C) \Big\} \\ \\
&\leq& \max \{ \dim  F_\emptyset, \ \dim C\}
\end{eqnarray*}
and so the formula holds trivially.  Thus, studying the Hausdorff and packing dimensions of inhomogeneous attractors is equivalent to studying the Hausdorff and packing dimensions of the corresponding homogeneous attractor, and thus is not an interesting problem in its own right.  However, the upper and lower box dimensions are not countably stable and so computing these dimensions in the inhomogeneous case is interesting and (perhaps) challenging, although one may still expect, somewhat na\"ively, that the relationship (\ref{inhomexpected}) should hold for these dimensions.  Recall that the lower and upper box dimensions of a bounded set $F \subseteq \mathbb{R}^n$ are defined by
\[
\underline{\dim}_\text{B} F = \liminf_{\delta \to 0} \, \frac{\log N_\delta (F)}{-\log \delta}
\qquad
\text{and}
\qquad
\overline{\dim}_\text{B} F = \limsup_{\delta \to 0} \,  \frac{\log N_\delta (F)}{-\log \delta},
\]
respectively, where $N_\delta (F)$ is the smallest number of sets required for a $\delta$-cover of $F$.  If $\underline{\dim}_\text{B} F = \overline{\dim}_\text{B} F$, then we call the common value the box dimension of $F$ and denote it by $\dim_\text{B} F$.  It is useful to note that we can replace $N_\delta$ with a myriad of different definitions all based on covering or packing the set at scale $\delta$, see \cite[Section 3.1]{falconer}.
\\ \\
The upper box dimensions of inhomogeneous self-similar sets were studied by Olsen and Snigireva \cite{olseninhom, ninaphd} where they proved that the expected relationship (\ref{inhomexpected}) is satisfied for upper box dimension assuming some strong separation conditions.  This result was then generalised by Fraser \cite{me_inhom}, where he proved the following bounds without assuming any separation conditions.  For any inhomogeneous self-similar set $F_C$, we have
\[
\max \{ \overline{\dim}_\text{B} F_\emptyset, \ \overline{\dim}_\text{B} C\} \ \leq \ \overline{\dim}_\text{B} F_C  \ \leq \ \max \{ s, \ \overline{\dim}_\text{B} C\},
\]
where $s$ is the \emph{similarity dimension} of $F_\emptyset$.  This shows that (\ref{inhomexpected}) is satisfied for upper box dimension provided the box dimension of the homogeneous attractor is given by the similarity dimension - which occurs if the IFS satisfies the open set condition (OSC), for example.  The second main result in \cite{me_inhom} was that (\ref{inhomexpected}) is in general \emph{not} satisfied for \emph{lower} box dimension.  It is evident in \cite{me_inhom} that proving upper bounds for the lower box dimension of an inhomogeneous attractor in general is a difficult problem.  In particular, if $N_\delta(C)$ oscillates wildly as $\delta \to 0$, then estimating $N_\delta( \mathcal{O})$ is tricky because it involves covering $C$ at many different scales. \emph{Covering regularity exponents} were introduced to try to tackle this problem, but yielded only upper and lower estimates.

\begin{figure}[H]
	\centering
	\includegraphics[width=135mm]{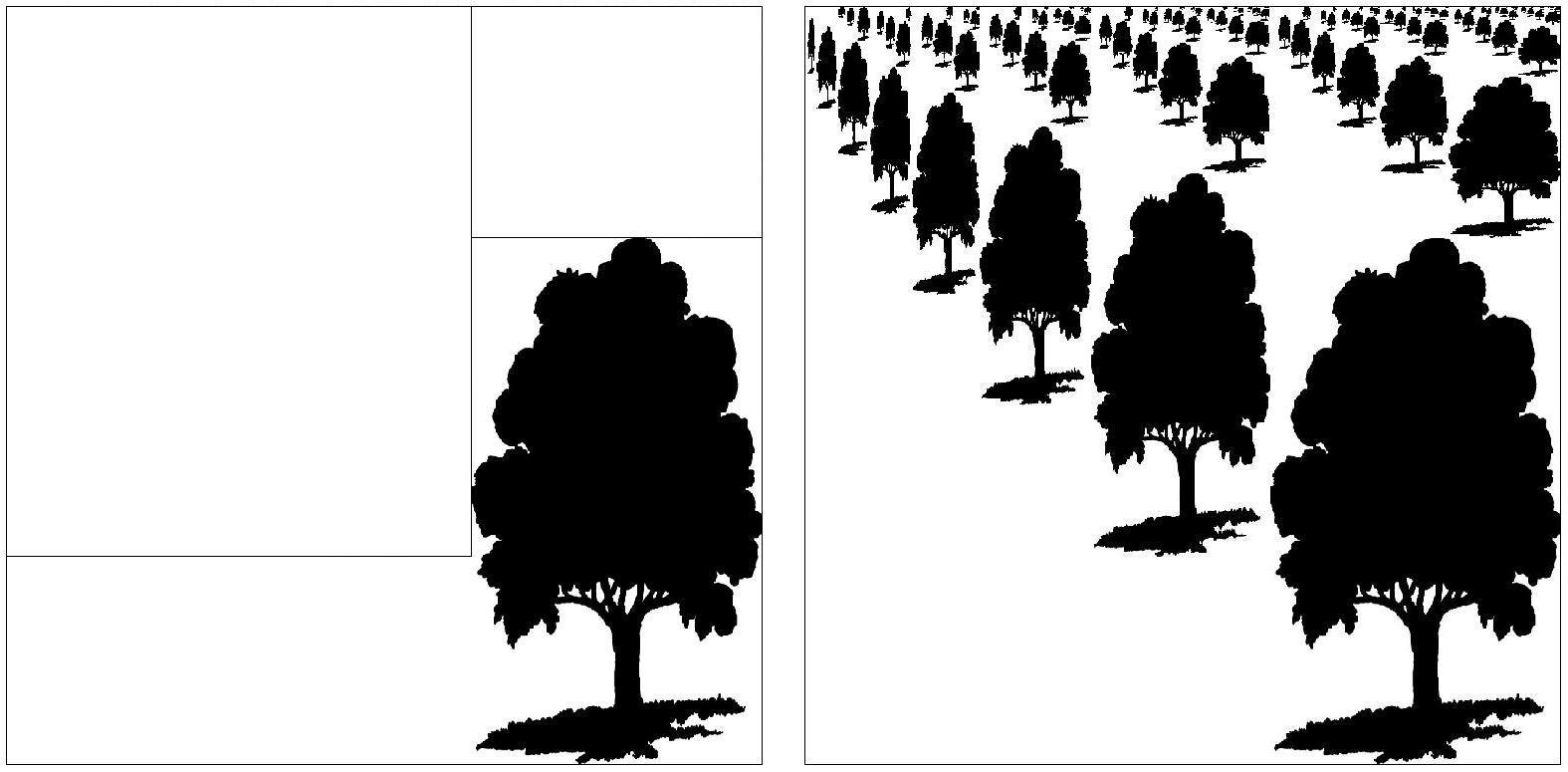}
	\caption{\emph{A fractal forest} (right).  The `forest' is depicted by an inhomogeneous attractor with the large tree at the bottom right being the condensation set.  The image on the left shows the two (affine) contractions used in the IFS, which, along with the condensation set, is the only data required to define the entire forest.  This is an example of \emph{image compression}.  The corresponding homogeneous attractor for this system is just a horizontal line at the top of the square.
}
\end{figure}

\subsection{Self-affine carpets} \label{affinecarpets}

Self-affine carpets are an important and well-studied class of fractal set.  This line of research began with the Bedford-McMullen carpets, introduced in the mid 80s \cite{bedford, mcmullen} and has blossomed into a vast and wide ranging field.  Particular attention has been paid to computing the dimensions of an array of different classes of self-affine carpets, see \cite{baranski, bedford, fengaffine, me_box, lalley-gatz, mackay, mcmullen}.  One reason that these classes are of interest is that they provide examples of self-affine sets with distinct Hausdorff and packing dimensions.  In this paper we will study inhomogeneous versions of the classes introduced by Lalley and Gatzouras in 1992 and Bara\'nski in 2007, and so we briefly recall these constructions and fix some notation.  Both classes consist of self-affine attractors of finite contractive iterated function systems acting on the unit square.
\\ \\
\textbf{Lalley-Gatzouras carpets:}
Take the unit square and divide it up into columns via a finite positive number of vertical lines.  Now divide each column up independently by slicing horizontally.  Finally select a subset of the small rectangles, with the only restriction being that the length of the base must be greater than or equal to the height, and for each chosen subrectangle include a map in the IFS which maps the unit square onto the rectangle via an orientation preserving linear contraction and a translation.  The Hausdorff and box dimensions of the attractors of such systems were computed in \cite{lalley-gatz}.  We note that in \cite{lalley-gatz} it was assumed that every map was strictly self-affine, i.e. that the length of the base of each rectangle must be strictly greater than the height.  We do not assume this here.

\begin{figure}[H]
	\centering
	\includegraphics[width=135mm]{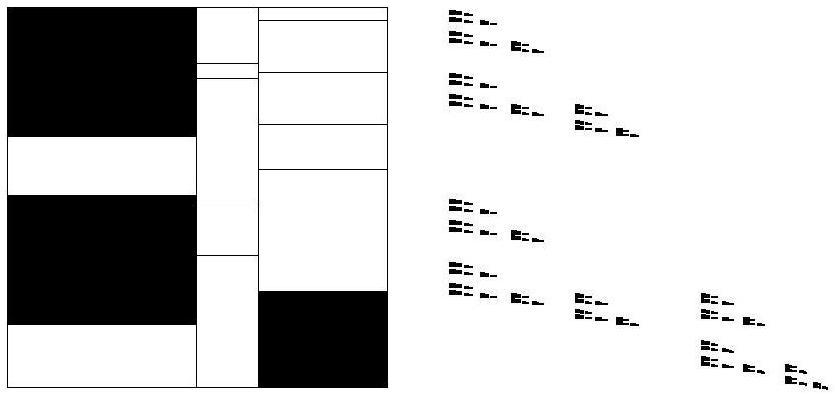}
\caption{The defining pattern for an IFS in the Lalley-Gatzouras class (left) and the corresponding attractor (right).}
\end{figure}

\textbf{Bara\'nski carpets:}
Again take the unit square, but this time divide it up into a collection of subrectangles by slicing horizontally and vertically a finite number of times (at least once in each direction).  Now take a subset of the subrectangles formed and form an IFS as above, choosing at least one rectangle with the horizontal side strictly longer than the vertical side and one with the vertical side strictly longer than the horizontal side.  The reason we assume this is so that the Baran\'nki class is disjoint from the Lalley-Gatzouras class.  Of course, if all the rectangles have longer vertical side, the IFS is equivalent to a Lalley-Gatzouras system via rotation by 90 degrees and so we omit discussion of this situation.  The key difference between the two classes is that in the Bara\'nski class one does not have that the strongest contraction is always in the vertical direction and this makes the Bara\'nski class significantly more difficult to deal with.  The Hausdorff and box dimensions of the attractors of such systems were computed in \cite{baranski}.  

\begin{figure}[H]
	\centering
	\includegraphics[width=135mm]{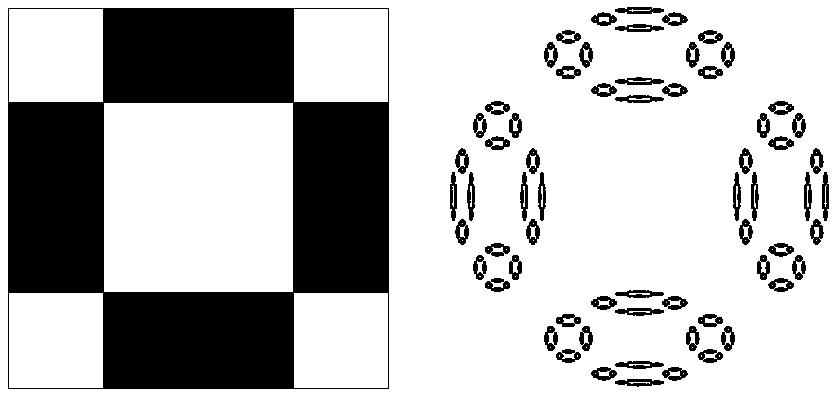}
\caption{The defining pattern for an IFS in the Bara\'nski class (left) and the corresponding attractor (right).}
\end{figure}

More general classes, containing both of the above, have been introduced and studied by Feng and Wang \cite{fengaffine} and Fraser \cite{me_box}.  See also the recent survey paper \cite{affinesurvey} which discusses the development of the study of self-affine carpets in the context of more general self-affine sets.
\\ \\
We will say that a set is an inhomogeneous self-affine carpet if $F_\emptyset$ is a self-affine carpet in the Lalley-Gatzouras or Bara\'nski class.   In order to state our results, we need to introduce some notation. Throughout this section $F_\emptyset$ will be a self-affine carpet which is the attractor of an IFS $\{S_i\}_{i \in \mathcal{I}}$ for some finite index set $\mathcal{I}$, with $\lvert \mathcal{I} \rvert \geq 2$ and, given a compact set $C\subseteq [0,1]^2$, $F_C$ will be the corresponding inhomogeneous self-affine carpet.  The maps $S_i$ in the IFS will be translate linear orientation preserving contractions on $[0,1]^2$ of the form
\[
S_i \big((x,y)\big) = (c_ix, d_iy)+\textbf{\emph{t}}_i
\]
for some contraction constants $c_i \in (0,1)$ in the horizontal direction and $d_i \in (0,1)$ in the vertical direction and a translation $\textbf{\emph{t}}_i \in \mathbb{R}^2$.

\section{Results}

In this section we state our results.  Let $\{S_i\}_{i \in \mathcal{I}}$ be an IFS described in Section \ref{affinecarpets} and fix a compact condensation set $C \subseteq [0,1]^2$.  As in the dimension theory of homogeneous self-affine carpets, the dimensions of orthogonal projections play an important role.  Let $\pi_1, \pi_2$ denote the orthogonal projections from the plane onto the first and second coordinates respectively and write
\[
s_1(F_\emptyset) = \dim_\text{B} \pi_1(F_\emptyset),
\]
\[
s_2(F_\emptyset) = \dim_\text{B} \pi_2(F_\emptyset),
\]
\[
\underline{s}_1(C) = \underline{\dim}_\text{B} \pi_1(C),
\]
\[
\overline{s}_1(C) = \overline{\dim}_\text{B} \pi_1(C),
\]
\[
\underline{s}_2(C) = \underline{\dim}_\text{B} \pi_2(C)
\]
and
\[
\overline{s}_2(C) = \overline{\dim}_\text{B} \pi_2(C).
\]
Note that $s_1(F_\emptyset)$ and $s_2(F_\emptyset)$ exist because they are the box dimensions of self-similar sets, whereas the equalities $\underline{s}_1(C) = \overline{s}_1(C)$ and $\underline{s}_2(C) = \overline{s}_2(C)$ may not hold, even if the box dimension of $C$ exists.  Finally, let $\underline{s}_A$,  $\overline{s}_A$, $\underline{s}_B$ and $\overline{s}_B$, be the unique solutions of
\[
\sum_{i \in \mathcal{I}} c_i^{\max\{s_1(F_\emptyset), \underline{s}_1(C)\} } d_i^{\underline{s}_A-\max\{s_1(F_\emptyset), \underline{s}_1(C)\} } = 1,
\]
\[
\sum_{i \in \mathcal{I}} c_i^{\max\{s_1(F_\emptyset), \overline{s}_1(C)\} } d_i^{\overline{s}_A-\max\{s_1(F_\emptyset), \overline{s}_1(C)\} } = 1,
\]
\[
\sum_{i \in \mathcal{I}} d_i^{\max\{s_2(F_\emptyset), \underline{s}_2(C)\} } c_i^{\underline{s}_B-\max\{s_2(F_\emptyset), \underline{s}_2(C)\} } = 1
\]
and
\[
\sum_{i \in \mathcal{I}} d_i^{\max\{s_2(F_\emptyset), \overline{s}_2(C)\} } c_i^{\overline{s}_B-\max\{s_2(F_\emptyset), \overline{s}_2(C)\} } = 1,
\]
respectively.  If $\underline{s}_A=\overline{s}_A$ or $\underline{s}_B=\overline{s}_B$, then write $s_A$ and $s_B$ respectively for the common values.  We need to make the following assumption to obtain a sharp formula for the upper box dimension of $F_C$.  Interestingly, this assumption concerns the Assouad dimension $\dim_\text{A}$ of projections of $C$.  This dimension is similar to upper box dimension, but is much more sensitive to local properties.  It is defined as
\newpage
\begin{eqnarray*}
\dim_\text{A} F \ = \  \inf \Bigg\{ &\alpha& : \text{     there exist constants $C, \, \rho>0$ such that,} \\
&\,& \text{ for all $0<r<R\leq \rho$, we have $\ \sup_{x \in F} \, N_r\big( B(x,R) \cap F \big) \ \leq \ C \bigg(\frac{R}{r}\bigg)^\alpha$ } \Bigg\}.
\end{eqnarray*}
See the papers \cite{me_assouad, luk} for more details.  The Assouad dimensions of homogeneous self-affine carpets were computed in \cite{me_assouad, mackay}.
\\ \\
\textbf{Assumption (A)}: $\dim_\text{A} \pi_i(C) \ \leq \ \max\{\overline{s}_i(C), s_i(F_\emptyset) \}$, for $i = 1,2$.
\\ \\
(A) is not a very strong assumption and is satisfied if, for example, $C$ is connected or the homogeneous IFS projects to an interval in the orthogonal directions.  We can now state our results.
\begin{thm} \label{main4}
Assume (A).  If $F_C$ is Lalley-Gatzouras, then
\[
\max\{ \underline{s}_A,  \underline{\dim}_\text{\emph{B}} C \} \ \leq \ \underline{\dim}_\text{\emph{B}} F_C \  \leq \ \overline{\dim}_\text{\emph{B}} F_C  \ =  \ \max\{ \overline{s}_A,  \overline{\dim}_\text{\emph{B}} C \}.
\]
If $F_C$ is Bara\'nski, then
\[
\max\{ \underline{s}_A, \underline{s}_B, \underline{\dim}_\text{\emph{B}} C \} \ \leq \ \underline{\dim}_\text{\emph{B}} F_C \  \leq \ \overline{\dim}_\text{\emph{B}} F_C  \ = \ \max\{ \overline{s}_A, \overline{s}_B,  \overline{\dim}_\text{\emph{B}} C \}.
\]
Furthermore, if we do not assume (A), then the same results are true but with the final equalities replaced by ``$\geq$'', i.e. we do not obtain an upper bound for $\overline{\dim}_\text{\emph{B}} F_C$.
\end{thm}
%IF NEEDED INCLUDE DESCRIPTION OF ASSOUAD DIMENSION
We will prove Theorem \ref{main4} in Section \ref{mainproof4}.  It is regrettable that we need to assume (A) to obtain a precise result but in a certain sense it is not important, because the highlight of this paper is the demonstration that the expected relationship (\ref{inhomexpected}) can fail for inhomogeneous carpets and, since we only need assumption (A) for the \emph{upper bound}, this does not change the situations where we can demonstrate this failure.  Also, we note that without assumption (A) our methods would yield an upper bound for $\overline{\dim}_\text{B} F_C$ where we use the Assouad dimensions of projections instead of upper box dimensions in the definition of $\overline{s}_A$ and $\overline{s}_B$, but we omit further discussion of this.  Indeed, it was our initial belief that (A) was not required at all and in fact Theorem \ref{main4} remains true without it, but surprisingly this is not the case.  In Section \ref{example2} we prove this by example.
\\ \\
Notice that we obtain a precise formula for the upper box dimension, but only lower bounds for the lower box dimension.  As discussed above, computing upper bounds for lower box dimension is difficult, even in the simpler setting of self-similar sets.  To obtain better estimates here, one could analyse the behaviour of the oscillations of the function $ \delta \mapsto N_\delta(C)$ using covering regularity exponents which were introduced in \cite{me_inhom} for example, but we do not pursue this and instead focus more on the upper box dimension and, in particular, the fact that the relationship (\ref{inhomexpected}) can fail. The following corollaries are immediate and include some simple sufficient conditions for the relationship (\ref{inhomexpected}) to hold, or not hold.
\begin{cor}
Suppose the box dimensions of $C$ and the orthogonal projections of $C$ exist and assume (A).  If $F_C$ is Lalley-Gatzouras, then
\[
\dim_\text{\emph{B}} F_C \  =  \  \max\{ s_A,  \dim_\text{\emph{B}} C \}.
\]
If $F_C$ is Bara\'nski, then
\[
\dim_\text{\emph{B}} F_C \  =  \  \max\{ s_A, s_B,  \dim_\text{\emph{B}} C \}.
\]
\end{cor}

\vspace{3mm}

\begin{cor}
Assuming (A), the relationship (\ref{inhomexpected}) holds for upper box dimension, i.e. 
\[
\overline{\dim}_\text{\emph{B}} F_C \  =  \  \max\{ \overline{\dim}_\text{\emph{B}} F_\emptyset,  \overline{\dim}_\text{\emph{B}} C \},
\]
in each of the following cases:
\begin{itemize}
\item[(1)] If $F_C$ is Lalley-Gatzouras and $\overline{s}_1(C) \leq s_1(F_\emptyset)$
\item[(2)] If $F_C$ is Bara\'nski, $\overline{s}_1(C) \leq s_1(F_\emptyset)$ and $\overline{s}_2(C) \leq s_2(F_\emptyset)$
\end{itemize}
\end{cor}

\vspace{3mm}

\begin{cor} \label{inhomfailscor}
The relationship (\ref{inhomexpected}) fails for lower box dimension, i.e. 
\[
\underline{\dim}_\text{\emph{B}} F_C \  >  \  \max\{ \underline{\dim}_\text{\emph{B}} F_\emptyset,  \underline{\dim}_\text{\emph{B}} C \},
\]
in each of the following cases:
\begin{itemize}
\item[(1)] If $F_C$ is Lalley-Gatzouras, $\underline{s}_A \geq \underline{\dim}_\text{\emph{B}} C$ and $\underline{s}_1(C) > s_1(F_\emptyset)$
\item[(2)] If $F_C$ is Bara\'nski, $\max\{\underline{s}_A, \underline{s}_A\} \geq \underline{\dim}_\text{\emph{B}} C$, $\underline{s}_1(C) > s_1(F_\emptyset)$ and $\underline{s}_2(C) > s_2(F_\emptyset)$.
\end{itemize}
Similarly, the relationship (\ref{inhomexpected}) fails for upper box dimension, i.e. 
\[
\overline{\dim}_\text{\emph{B}} F_C \  >  \  \max\{ \overline{\dim}_\text{\emph{B}} F_\emptyset,  \overline{\dim}_\text{\emph{B}} C \},
\]
in each of the following cases:
\begin{itemize}
\item[(1)] If $F_C$ is Lalley-Gatzouras, $\overline{s}_A \geq \overline{\dim}_\text{\emph{B}} C$ and $\overline{s}_1(C) > s_1(F_\emptyset)$
\item[(2)] If $F_C$ is Bara\'nski, $\max\{\overline{s}_A, \overline{s}_A\} \geq \overline{\dim}_\text{\emph{B}} C$, $\overline{s}_1(C) > s_1(F_\emptyset)$ and $\overline{s}_2(C) > s_2(F_\emptyset)$.
\end{itemize}
\end{cor}

In a certain sense, Corollary \ref{inhomfailscor} is the most interesting as it gives simple, and easily constructible, conditions for the relationship (\ref{inhomexpected}) to fail.  We will construct such an example in Section \ref{combs}.
\\ \\
Although the underlying homogeneous IFSs automatically satisfy the OSC, it is worth remarking that our results impose no further separation conditions concerning the condensation set $C$.  In particular, $C$ may have arbitrary overlaps with $F_\emptyset$.
\\ \\
It would be interesting to extend the results in this case to the more general carpets introduced by Feng and Wang \cite{fengaffine} or Fraser \cite{me_box}.  However, there are some additional difficulties in these cases.  Indeed, the Feng-Wang case is intimately related to the question of whether the relationship (\ref{inhomexpected}) holds for self-similar sets not satisfying the OSC, see \cite[Question 2.4]{me_inhom}.  In particular, the sets $\pi_1(F_C)$ and $\pi_2(F_C)$ are inhomogeneous self-similar sets and knowledge of their dimension is crucial in the subsequent proofs.  Furthermore, in the Fraser case, one would need to extend the results on inhomogeneous self-similar sets to the graph-directed case.  There is certainly scope for future research here, and it is easily seen that our methods give solutions to the more general problem in certain cases and can always provide non-trivial estimates; however, we omit further discussion.

\section{Examples}

\subsection{Inhomogeneous fractal combs} \label{combs}

We give a construction of an inhomogeneous Bedford-McMullen carpet, which we refer to as an \emph{inhomogeneous fractal comb} which exhibits some interesting properties.  The underlying homogeneous IFS will be a Bedford-McMullen construction where the unit square has been divided into 2 columns of width $1/2$, and $n > 2$ rows of height $1/n$.  The IFS is then made up of all the maps which correspond to the left hand column.  The condensation set for this construction is taken as $C = [0,1] \times \{0\}$, i.e. the base of the unit square.  The inhomogeneous attractor is termed the inhomogeneous fractal comb and is denoted by $F_C^n$.
\\ \\
It follows from Theorem \ref{main4} that $\underline{\dim}_\text{B} F_C^n = \overline{\dim}_\text{B} F_C^n$ is the unique solution of
\[
n \, 2^{-1} \, n^{1-s} = 1
\]
which gives
\[
\underline{\dim}_\text{B} F_C^n \ = \ \overline{\dim}_\text{B} F_C^n  \ = \ 2-\log 2/ \log n \ > \ 1.
\]
However,
\[
\max \{ \overline{\dim}_\text{B} F_\emptyset, \ \overline{\dim}_\text{B} C \} \  = \ 1
\]
and thus our fractal comb provides a simple example showing that the `expected relationship' for upper box dimension (\ref{inhomexpected}) can fail for self-affine sets, even if the homogeneous IFS satisfies the OSC.  This is in stark contrast to the self-similar setting, see \cite[Corollary 2.2]{me_inhom}.
\\ \\
This example has another interesting property: it shows that $\overline{\dim}_\text{B} F_C^n$ does not just depend on the sets $F_\emptyset$ and $C$, but also depends on the IFS itself.  To see this observe that $F_\emptyset$ and $C$ do not depend on $n$, but $\overline{\dim}_\text{B} F_C^n$ does.  In fact $F_\emptyset = \{0 \} \times [0,1]$, i.e. the left hand side of the unit square, for any $n$.  Again, this behaviour is not observed in the self-similar setting.
\\ \\
Finally, observe that, although the inhomogeneous fractal combs are subsets of $\mathbb{R}^2$ and the expected box dimension is 1, we can find examples where the achieved box dimension is arbitrarily close to 2, demonstrating that, in this case, there is no limit to how `badly' the relationship (\ref{inhomexpected}) can fail.
\begin{figure}[H]
	\centering
	\includegraphics[width=140mm]{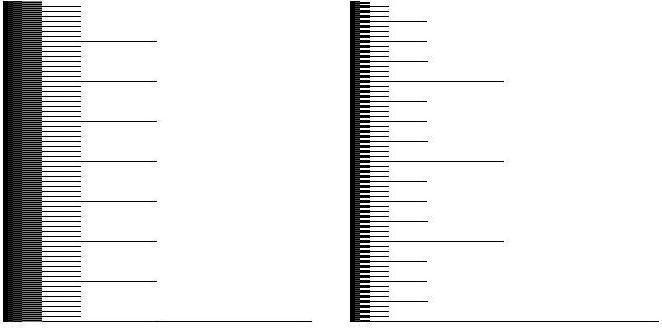}
	\caption{Two fractal combs: the inhomogeneous fractal combs $F_C^8$, with box dimension $5/3$ (left); and $F_C^4$, with box dimension $3/2$ (right).}
\end{figure}

\subsection{A counterexample} \label{example2}

In this section we provide an example showing that the lower bound given in Theorem \ref{main4} for the upper box dimension is not sharp in general.  This is somewhat surprising and relies on the condensation set $C$ having strange scaling properties.  The specific construction for $C$ used here was suggested to us by Tuomas Orponen, for which we are grateful.
\\ \\
Take the unit square and fix $\eta \in (0,1]$ which is the reciprocal of a power of 4. The reason 4 is chosen is to make the subsequent construction fit together with the dyadic grid. Down the right hand side of the square place $\eta^{-1/2}$ smaller squares of side length $\eta^{1/2}$ on top of each other forming a single column.  Inside each of these smaller squares place $\eta^{-1/2}$ even smaller squares of side length $\eta$ along the base.  This is the basic defining pattern.  Now to construct $C$, iterate this construction, but at each stage choose $\eta$ drastically smaller than the $\eta$ at the previous stage.  More precisely, choose $\eta_1 = 1/4$ and then define $\eta_i$ recursively by $\eta_i = (\eta_1 \cdots \eta_{i-1})^i$.  The limit set is a compact set $C \subset [0,1]^2$ and elementary calculations yield that
\[
\overline{\dim}_\text{B} C = 1 \qquad \text{and} \qquad \overline{\dim}_\text{B} \pi_1(C) = 1/2.
\]

\begin{figure}[H]
	\centering
	\includegraphics[width=148mm]{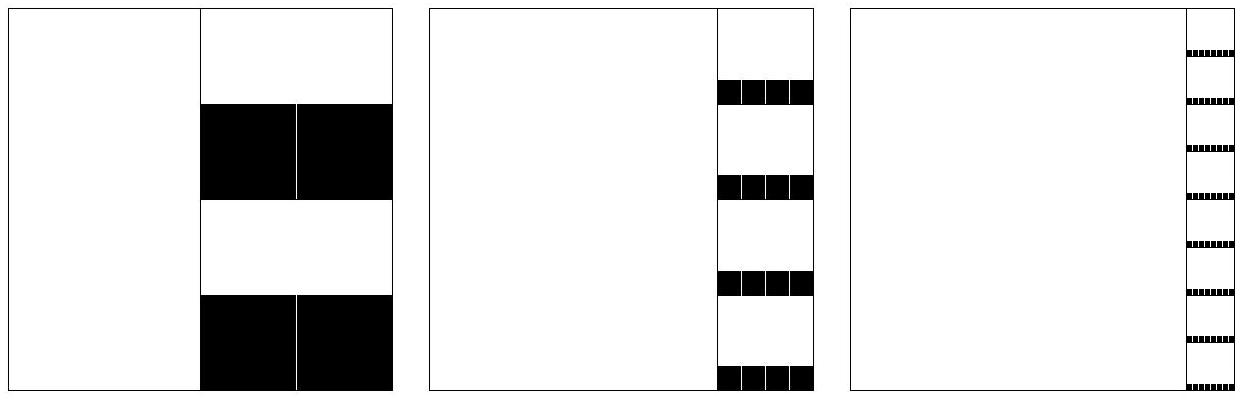}
	\caption{The defining patterns with $\eta = 1/4$, $1/16$ and $1/64$ respectively.}
\end{figure}

Let $\hat \eta_i  = \eta_1 \cdots \eta_i$ and observe that at the $i$th stage of the construction there are $\hat \eta_i ^{-1}$ squares of side length $\hat \eta_i$.  The key property which we will utilise is the following.  For $\alpha \in (0,1]$, define $T_\alpha: [0,1]^2 \to [0,1]^2$ by $T_\alpha(x,y) = (x, \alpha y)$.  For all $i \in \mathbb{N}$ we have, using the dyadic grid definition of $N_\delta$,
\begin{equation} \label{goodscaling}
N_{\hat \eta_i}\Big( T_{ \eta_i^{1/2}}(C)\Big) \geq \hat \eta_i^{-1}.
\end{equation}
To see this consider the $\hat \eta_i ^{-1}$ squares of side length $\hat \eta_i$ present at the $i$th stage of construction and note that when we scale vertically by $\eta_i^{1/2}$ each $\hat \eta_i$ square still contributes at least one to $N_{\hat \eta_i}$.
\\ \\
The homogeneous IFS used for this example will be the same as for the fractal comb $F_C^4$, i.e., a Bedford-McMullen carpet with $m=2$, $n=4$ and four mappings used, all corresponding to the left column.  Again, it is clear that
\[
\overline{\dim}_\text{B} F_\emptyset = 1 \qquad \text{and} \qquad \overline{\dim}_\text{B} \pi_1(F_\emptyset) = 0.
\]
As such, the lower bound for $\overline{\dim}_\text{B} F_C$ given by Theorem \ref{main4} is the solution $s$ of
\[
4 \, (1/2)^{1/2} \, (1/4)^{s-1/2} = 1
\]
which yields $\overline{\dim}_\text{B} F_C \geq s = 5/4$.  Note that this is already greater than the value given by (\ref{inhomexpected}) which is 1.  We now claim that in fact $\overline{\dim}_\text{B} F_C \geq 4/3$.
\\ \\
To see this, let $\hat \eta = \hat \eta_i$ and $\eta =  \eta_i$ for some $i$ and choose $k  \in \mathbb{N}$ such that $2^k = \eta^{-1/2}$ and let $\delta =\hat \eta \, \eta^{1/2}$.  By the definition of the $\eta_i$ we note $\eta = \hat \eta^{1-1/(i+1)}$.  We have
\begin{eqnarray*}
N_\delta(F_C) &\geq& \sum_{\mathcal{I}^k} N_\delta \big(S_{\textbf{\emph{i}}}(C) \big) \\ \\
&=& 4^k N_{\delta 2^k} \big(T_{2^{-k}}(C) \big) \qquad \text{by scaling up by $2^k$}  \\ \\
&=&\eta^{-1} N_{\hat \eta} \big(T_{\eta^{1/2}}(C) \big) \qquad \text{by the choice of $k$} \\ \\
&\geq&  \hat \eta^{-2+1/i} \qquad \text{by (\ref{goodscaling})}\\ \\
&=& \delta^{-\alpha}
\end{eqnarray*}
where
\[
\alpha \ = \ \frac{2-1/(i+1)}{3/2-1/(2i+2)} \ \to \  4/3 \qquad \text{as $i \to \infty$}
\]
and so letting $\delta$ tend to zero through the values $\hat \eta_i \, \eta_i^{1/2}$ proves the claim.  Of course, (A) is not satisfied for this construction.  Indeed, we note that $\dim_\text{A}\pi_1(C) = 1$, which can be shown by constructing a weak tangent to $\pi_1(C)$ at 1, see \cite[Proposition 2.1]{mackay}.

\begin{comment}

\subsection{Another example}

In this section we provide an example with a slightly more exotic looking structure.  Despite this, however, it is perhaps less interesting than the previous one as the relationship (\ref{inhomexpected}) holds.  The underlying homogeneous IFS will be in the Bara\'nski class and will consist of the mappings
\[
S_1 = \left( \begin{array}{cc}
\tfrac{3}{10} & 0\\ 
0& \tfrac{3}{10}
 \end{array} \right),
\]
\[
S_2 = \left( \begin{array}{cc}
\tfrac{3}{10} & 0\\
0& \tfrac{7}{10}
 \end{array} \right) +  \left( \begin{array}{c}
0\\
\tfrac{3}{10}
 \end{array} \right)
\]
and
\[
S_3 = \left( \begin{array}{cc}
\tfrac{7}{10} & 0\\
0& \tfrac{3}{10}
 \end{array} \right) +  \left( \begin{array}{c}
\tfrac{3}{10}\\
0
 \end{array} \right)
\]
and the condensation set will be the Sierpi\'nski triangle.

\begin{figure}[H]
	\centering
	\includegraphics[width=135mm]{inhomcarpet2}
	\caption{The homogeneous Bara\'nski type carpet described above (left) and the corresponding inhomogeneous carpet with condensation set based on the Sierpi\'nski triangle (right).}
\end{figure}

One can easily see that $s_A = s_B = s$ is the solution of
\[
\big(\tfrac{3}{10}\big)\big(\tfrac{3}{10}\big)^{s-1}+\big(\tfrac{3}{10}\big)\big(\tfrac{7}{10}\big)^{s-1}+\big(\tfrac{7}{10}\big)\big(\tfrac{3}{10}\big)^{s-1} = 1
\]
which is $s = 1.2647\dots$ and it follows from Theorem \ref{main4} that
\[
\underline{\dim}_\text{B} F_C = \overline{\dim}_\text{B} F_C = \max\{ s, \dim_\text{B} C \} = \max\{s , \log3/\log2 \} = \frac{\log3}{\log2} = 1.5849 \dots
\]
\end{comment}

\section{Proofs} \label{mainproof4}

\subsection{Preliminary results}

In this section we will introduce some notation and establish some simple estimates before beginning the main proofs.  For $\textbf{\emph{i}} =  (i_1, i_2, \dots, i_k ) \in \mathcal{I}^*$, let $c_\textbf{\emph{i}} = c_{i_1} \cdots c_{i_k}$, $d_\textbf{\emph{i}} = d_{i_1} \cdots d_{i_k}$, let $\alpha_1(\textbf{\emph{i}}) \geq \alpha_2(\textbf{\emph{i}})$ be the singular values of the map $S_\textbf{\emph{i}}$, i.e. $\alpha_1(\textbf{\emph{i}}) = \max\{ c_\textbf{\emph{i}}, d_\textbf{\emph{i}}\}$ and $\alpha_2(\textbf{\emph{i}}) = \min\{ c_\textbf{\emph{i}}, d_\textbf{\emph{i}}\}$, and $\alpha_{\min} = \min_{i \in \mathcal{I}} \alpha_2(i) \in (0,1)$.  Also, write
\[
\pi_{\textbf{\emph{i}}} = \left\{ \begin{array}{cc}
\pi_1 &  \text{if $c_\textbf{\emph{i}} \geq d_\textbf{\emph{i}}$}\\ 
\pi_2 &   \text{if $c_\textbf{\emph{i}} < d_\textbf{\emph{i}}$}
\end{array} \right. 
\]
\[
s_{\textbf{\emph{i}}}(F_\emptyset) = \left\{ \begin{array}{cc}
s_1(F_\emptyset)  &  \text{if $c_\textbf{\emph{i}} \geq d_\textbf{\emph{i}}$}\\ 
s_2(F_\emptyset)  &   \text{if $c_\textbf{\emph{i}} < d_\textbf{\emph{i}}$}
\end{array} \right. 
\]
\[
\overline{s}_{\textbf{\emph{i}}}(C) = \left\{ \begin{array}{cc}
\overline{s}_1(C)   &  \text{if $c_\textbf{\emph{i}} \geq d_\textbf{\emph{i}}$}\\ 
\overline{s}_2(C)  &   \text{if $c_\textbf{\emph{i}} < d_\textbf{\emph{i}}$}
\end{array} \right. 
\]
\[
\underline{s}_{\textbf{\emph{i}}}(C) = \left\{ \begin{array}{cc}
\underline{s}_1(C)   &  \text{if $c_\textbf{\emph{i}} \geq d_\textbf{\emph{i}}$}\\ 
\underline{s}_2(C)  &   \text{if $c_\textbf{\emph{i}} < d_\textbf{\emph{i}}$}
\end{array} \right. 
\]
The sets $\pi_1(\overline{\mathcal{O}})$ and $\pi_2(\overline{\mathcal{O}})$ are inhomogeneous self-similar sets with condensation sets $\pi_1(C)$ and $\pi_2(C)$ respectively.  The underlying IFSs (derived in the obvious way from the original IFS) satisfy the OSC and so it follows from \cite[Corollary 2.2]{me_inhom} that
\[
 \max\{s_1(F_\emptyset), \underline{s}_1(C)\} \ \leq \ \underline{\dim}_\text{B}\pi_1(\mathcal{O}) \ \leq \ \overline{\dim}_\text{B}\pi_1(\mathcal{O}) \ = \ \max\{s_1(F_\emptyset), \overline{s}_1(C)\}
\]
and
\[
 \max\{s_2(F_\emptyset), \underline{s}_2(C)\} \ \leq \ \underline{\dim}_\text{B}\pi_2(\mathcal{O}) \ \leq \ \overline{\dim}_\text{B}\pi_2(\mathcal{O}) \ = \ \max\{s_2(F_\emptyset), \overline{s}_2(C)\}.
\]
It follows that, for all $\varepsilon \in (0,1]$, there exists a constant $C_\varepsilon \geq 1$ such that for all $\textbf{\emph{i}} \in \mathcal{I}^*$ and all $\delta \in (0,\alpha_{\min}^{-1}]$ we have
\begin{equation} \label{boxorbproj}
C_\varepsilon^{-1} \delta^{ -\max\{ \underline{s}_\textbf{\emph{i}}(C), s_\textbf{\emph{i}}(F_\emptyset)\} + \varepsilon} \ \leq \ N_\delta \big(\pi_\textbf{\emph{i}}(\mathcal{O}) \big) \ \leq \ C_\varepsilon \delta^{-\max\{\overline{s}_\textbf{\emph{i}}(C), s_\textbf{\emph{i}}(F_\emptyset)\} - \varepsilon}
\end{equation}

For $\textbf{\emph{i}} = (i_1,i_2, \dots, i_{k-1}, i_k)  \in \mathcal{I}^*$, let $\overline{\textbf{\emph{i}}} = (i_1,i_2, \dots, i_{k-1}) \in \mathcal{I}^* \cup \{ \omega\}$, where $\omega$ is the empty word.  For notational convenience the map $S_\omega$ is taken to be the identity map.For $\delta \in (0,1]$ we define the $\delta$-\emph{stopping}, $\mathcal{I}_\delta$, as follows:
\[
\mathcal{I}_\delta = \big\{\textbf{\emph{i}} \in \mathcal{I}^* : \alpha_2(\textbf{\emph{i}}) < \delta \leq \alpha_2( \overline{\textbf{\emph{i}}}) \big\}.
\]
Note that for $\textbf{\emph{i}} \in \mathcal{I}_\delta$ we have
\begin{equation} \label{stoppingest}
\alpha_{\min} \, \delta \leq \alpha_2(\textbf{\emph{i}}) < \delta.
\end{equation}

\subsection{Proof of the lower bound for the lower box dimension in Theorem \ref{main4}}

In this section we will prove that if $F$ is in the Bara\'nski class, then $\max\{ \underline{s}_A, \underline{s}_B, \underline{\dim}_\text{B} C \}  \leq  \underline{\dim}_\text{B} F_C$.  The proof of the analogous inequality for the Lalley-Gatzouras class is similar and omitted.  Since lower box dimension is monotone, it suffices to show that $\max\{ \underline{s}_A, \underline{s}_B \}  \leq  \underline{\dim}_\text{B} F_C$, and we will assume without loss of generality that $\max\{ \underline{s}_A, \underline{s}_B \} = \underline{s}_A$.
\\ \\
Let $\varepsilon \in (0, \underline{s}_A)$, $\delta \in (0,1]$ and $U$ be any closed square of side length $\delta$.  Also, let
\[
M = \min\big\{n \in \mathbb{N} : n \geq \alpha_{\min}^{-1}+2\big\}.
\]
Since $\{S_\textbf{\emph{i}}\big([0,1]^2\big)\}_{\textbf{\emph{i}}\in \mathcal{I}_\delta}$ is a collection of pairwise disjoint open rectangles each with shortest side having length at least $\alpha_{\min} \delta$, it is clear that $U$ can intersect no more than $M^2$ of the sets $\{ S_i(\mathcal{O})\}_{\textbf{\emph{i}}\in \mathcal{I}_\delta}$ since $S_i(\mathcal{O}) \subseteq S_\textbf{\emph{i}}\big([0,1]^2\big)$ for all $\textbf{\emph{i}}\in \mathcal{I}_\delta$.  It follows that, using the $\delta$-mesh definition of $N_\delta$, we have
\[
\sum_{\textbf{\emph{i}}\in \mathcal{I}_\delta} N_\delta \big(S_\textbf{\emph{i}}(\mathcal{O})\big) \  \leq  \ M^2 \, N_\delta \Bigg(\bigcup_{\textbf{\emph{i}}\in \mathcal{I}_\delta} S_\textbf{\emph{i}}(\mathcal{O}) \Bigg)  \ \leq \  M^2 \, N_\delta(\mathcal{O}).
\]
This yields
\begin{eqnarray*}
N_\delta(\mathcal{O}) & \geq & \tfrac{1}{M^2} \  \sum_{\textbf{\emph{i}} \in \mathcal{I}_\delta} N_\delta \big(S_\textbf{\emph{i}}(\mathcal{O})\big) \\ \\
& = & \tfrac{1}{M^2} \   \sum_{\textbf{\emph{i}} \in \mathcal{I}_\delta} N_{\delta/\alpha_1(\textbf{\emph{i}})} \big(\pi_\textbf{\emph{i}}(\mathcal{O})\big)  \qquad \qquad \text{since $\alpha_2(\textbf{\emph{i}}) < \delta$}\\ \\
& \geq & \tfrac{1}{M^2} \   \sum_{\textbf{\emph{i}} \in \mathcal{I}_\delta} C_\varepsilon^{-1} \,\bigg( \frac{\alpha_1(\textbf{\emph{i}})}{\delta} \bigg)^{\max\{\underline{s}_\textbf{\emph{i}}(C), s_\textbf{\emph{i}}(F_\emptyset) \}-\varepsilon}  \qquad \qquad \text{by (\ref{boxorbproj})}\\ \\
& = & \tfrac{1}{M^2C_\varepsilon} \   \delta^{-\underline{s}_A+\varepsilon} \ \sum_{\textbf{\emph{i}} \in \mathcal{I}_\delta} \alpha_1(\textbf{\emph{i}})^{\max\{\underline{s}_\textbf{\emph{i}}(C), s_\textbf{\emph{i}}(F_\emptyset) \}} \delta^{\underline{s}_A - \max\{\underline{s}_\textbf{\emph{i}}(C), s_\textbf{\emph{i}}(F_\emptyset) \}} \\ \\
& \geq & \tfrac{1}{M^2C_\varepsilon} \   \delta^{-\underline{s}_A+\varepsilon} \ \sum_{\textbf{\emph{i}} \in \mathcal{I}_\delta} \alpha_1(\textbf{\emph{i}})^{\max\{\underline{s}_\textbf{\emph{i}}(C), s_\textbf{\emph{i}}(F_\emptyset) \}} \alpha_2(\textbf{\emph{i}})^{\underline{s}_A - \max\{\underline{s}_\textbf{\emph{i}}(C), s_\textbf{\emph{i}}(F_\emptyset) \}}
\end{eqnarray*}
by (\ref{stoppingest}).  We now claim that for all $\textbf{\emph{i}}\in \mathcal{I}_\delta$ we have
\[
\alpha_1(\textbf{\emph{i}})^{\max\{\underline{s}_\textbf{\emph{i}}(C), s_\textbf{\emph{i}}(F_\emptyset) \}} \alpha_2(\textbf{\emph{i}})^{\underline{s}_A - \max\{\underline{s}_\textbf{\emph{i}}(C), s_\textbf{\emph{i}}(F_\emptyset) \}} \ \geq \ c_\textbf{\emph{i}}^{\max\{\underline{s}_1(C), s_1(F_\emptyset) \}} d_\textbf{\emph{i}}^{\underline{s}_A - \max\{\underline{s}_1(C), s_1(F_\emptyset) \}}.
\]
If $c_\textbf{\emph{i}} \geq d_\textbf{\emph{i}}$, then we trivially have equality, so assume that $c_\textbf{\emph{i}} < d_\textbf{\emph{i}}$, in which case
\newpage
\begin{eqnarray*}
\alpha_1(\textbf{\emph{i}})^{\max\{\underline{s}_\textbf{\emph{i}}(C), s_\textbf{\emph{i}}(F_\emptyset) \}} \alpha_2(\textbf{\emph{i}})^{\underline{s}_A - \max\{\underline{s}_\textbf{\emph{i}}(C), s_\textbf{\emph{i}}(F_\emptyset) \}} &=& d_\textbf{\emph{i}}^{\max\{\underline{s}_2(C), s_2(F_\emptyset) \}} c_\textbf{\emph{i}}^{\underline{s}_A - \max\{\underline{s}_2(C), s_2(F_\emptyset) \}} \\ \\
&=& c_\textbf{\emph{i}}^{\max\{\underline{s}_1(C), s_1(F_\emptyset) \}} d_\textbf{\emph{i}}^{\underline{s}_A - \max\{\underline{s}_1(C), s_1(F_\emptyset) \}} \\ \\
&\,& \qquad \quad \cdot \bigg( \frac{d_\textbf{\emph{i}}}{c_\textbf{\emph{i}}}\bigg)^{ \max\{\underline{s}_1(C), s_1(F_\emptyset) \} + \max\{\underline{s}_2(C), s_2(F_\emptyset) \} -\underline{s}_A } \\ \\
&\geq& c_\textbf{\emph{i}}^{\max\{\underline{s}_1(C), s_1(F_\emptyset) \}} d_\textbf{\emph{i}}^{\underline{s}_A - \max\{\underline{s}_1(C), s_1(F_\emptyset) \}}
\end{eqnarray*}
since it is easily seen that $\underline{s}_A \leq  \max\{\underline{s}_1(C), s_1(F_\emptyset) \} + \max\{\underline{s}_2(C), s_2(F_\emptyset) \}$.  Combining this with the above estimate yields
\begin{eqnarray*}
N_\delta(\mathcal{O}) & \geq &  \tfrac{1}{M^2 C_\varepsilon } \  \delta^{-\underline{s}_A+\varepsilon} \sum_{\textbf{\emph{i}} \in \mathcal{I}_\delta } c_\textbf{\emph{i}}^{\max\{\underline{s}_1(C), s_1(F_\emptyset) \}} d_\textbf{\emph{i}}^{\underline{s}_A - \max\{\underline{s}_1(C), s_1(F_\emptyset) \}}\\ \\
& = & \tfrac{1}{M^2C_\varepsilon } \  \delta^{-(\underline{s}_A-\varepsilon)}
\end{eqnarray*}
by repeated application of the definition of $\underline{s}_A$.  This proves that $\underline{\dim}_\text{B} F_C = \underline{\dim}_\text{B} \mathcal{O} \geq  \underline{s}_A - \varepsilon$ and letting $\varepsilon$ tend to zero gives the desired lower bound. \hfill \qed

\subsection{Proof of the upper bound for the upper box dimension in Theorem \ref{main4}}

In this section we will prove that if $F_C$ is in the Bara\'nski class and satisfies assumption (A), then $\overline{\dim}_\text{B} F_C  \leq   \max\{ \overline{s}_A, \overline{s}_B,  \overline{\dim}_\text{B} C \}$.  The proof of the analogous inequality for the Lalley-Gatzouras class is similar and omitted.  Let $s = \max\{ \overline{s}_A, \overline{s}_B,  \overline{\dim}_\text{B} C \}$ and $\varepsilon>0$.  Since $F_C = \overline{\mathcal{O}}$ and upper box dimension is stable under taking closures, it suffices to estimate $\overline{\dim}_\text{B} \mathcal{O}$.  We have
\begin{eqnarray*}
N_\delta(\mathcal{O}) & = & N_\delta \Bigg( \bigcup_{\textbf{\emph{i}} \in \mathcal{I}_\delta} S_\textbf{\emph{i}}(\mathcal{O}) \ \cup \ \bigcup_{ \substack{\textbf{\emph{i}} \in \mathcal{I}^* \cup\{ \omega\}: \\ \\
\alpha_2(\textbf{\emph{i}}) > \delta}} S_\textbf{\emph{i}}(C) \Bigg)   \\ \\
& \leq &  \sum_{\textbf{\emph{i}} \in \mathcal{I}_\delta} N_\delta \big(S_\textbf{\emph{i}}(\mathcal{O}) \big) \ + \ \sum_{ \substack{\textbf{\emph{i}} \in \mathcal{I}^* \cup\{ \omega\}: \\ \\
\alpha_2(\textbf{\emph{i}}) > \delta}} N_\delta \big(S_\textbf{\emph{i}}(C) \big).
\end{eqnarray*}
We will analyse these two terms separately.  For the first term
\begin{eqnarray*}
 \sum_{\textbf{\emph{i}} \in \mathcal{I}_\delta} N_\delta \big(S_i(\mathcal{O})\big) & = &  \sum_{\textbf{\emph{i}} \in \mathcal{I}_\delta} N_{\delta/\alpha_1(\textbf{\emph{i}})} \big(\pi_\textbf{\emph{i}}(\mathcal{O})\big) \qquad \qquad \text{since $\alpha_2(\textbf{\emph{i}}) < \delta$} \\ \\
& \leq &   \sum_{\textbf{\emph{i}} \in \mathcal{I}_\delta} C_\varepsilon \bigg( \frac{\alpha_1(\textbf{\emph{i}})}{\delta} \bigg)^{\max\{\overline{s}_\textbf{\emph{i}}(C), s_\textbf{\emph{i}}(F_\emptyset) \}+\varepsilon} \qquad \qquad \text{by (\ref{boxorbproj})}\\ \\
& =&  C_\varepsilon \,  \delta^{-s-\varepsilon} \ \sum_{\textbf{\emph{i}} \in \mathcal{I}_\delta} \alpha_1(\textbf{\emph{i}})^{\max\{\overline{s}_\textbf{\emph{i}}(C), s_\textbf{\emph{i}}(F_\emptyset) \}} \delta^{s - \max\{\overline{s}_\textbf{\emph{i}}(C), s_\textbf{\emph{i}}(F_\emptyset) \}}\\ \\
& \leq &  C_\varepsilon \, \alpha_{\min}^{-2} \, \delta^{-s-\varepsilon} \ \sum_{\textbf{\emph{i}} \in \mathcal{I}_\delta} \alpha_1(\textbf{\emph{i}})^{\max\{\overline{s}_\textbf{\emph{i}}(C), s_\textbf{\emph{i}}(F_\emptyset) \}} \alpha_2(\textbf{\emph{i}})^{s - \max\{\overline{s}_\textbf{\emph{i}}(C), s_\textbf{\emph{i}}(F_\emptyset) \}} \qquad \qquad \text{by (\ref{stoppingest})} \\ \\
& \leq &  C_\varepsilon \,  \alpha_{\min}^{-2} \, \delta^{-s-\varepsilon} \ \Bigg( \sum_{\textbf{\emph{i}} \in \mathcal{I}_\delta} c_\textbf{\emph{i}}^{\max\{\overline{s}_1(C), s_1(F_\emptyset) \}} d_\textbf{\emph{i}}^{\overline{s}_A - \max\{\overline{s}_1(C), s_1(F_\emptyset) \}}\\ \\
&\,& \qquad \qquad \qquad \qquad + \ \sum_{\textbf{\emph{i}} \in \mathcal{I}_\delta} d_\textbf{\emph{i}}^{\max\{\overline{s}_2(C), s_2(F_\emptyset) \}} c_\textbf{\emph{i}}^{\overline{s}_B - \max\{\overline{s}_2(C), s_2(F_\emptyset) \}} \Bigg)\\ \\
& \leq & \  2 \, C_\varepsilon \,  \alpha_{\min}^{-2} \, \delta^{-(s+\varepsilon)}
\end{eqnarray*}
by repeated application of the definitions of $\overline{s}_A$ and $\overline{s}_B$.  The second term is awkward as we have to estimate $N_\delta \big(S_\textbf{\emph{i}}(C) \big)$ for $\textbf{\emph{i}}$ with various different values of $\alpha_2(\textbf{\emph{i}}) > \delta$.  This is the only occasion in the proof where we require assumption (A).
\begin{lma} \label{assumptionAlemma}
Assume (A), let $\varepsilon>0$ and let $\delta \in (0,1]$.  There exists a constant $D_\varepsilon>0$ such that for all $\textbf{i} \in \mathcal{I}^* \cup\{ \omega\}$ such that $\alpha_2(\textbf{i}) > \delta$, we have
\[
N_\delta \big(S_\textbf{i}(C) \big) \ \leq \ D_\varepsilon \,  \delta^{-(s+\varepsilon)} \, \alpha_1(\textbf{i})^{\max\{\overline{s}_\textbf{i}(C), s_\textbf{i}(F_\emptyset) \}} \alpha_2(\textbf{i})^{s+\varepsilon/2-\max\{\overline{s}_\textbf{i}(C), s_\textbf{i}(F_\emptyset) \}}
\]
\end{lma}

\begin{proof}
First take a cover of $C$ by fewer than
\[
D_{1,\varepsilon} \bigg(\frac{\alpha_2(\textbf{\emph{i}})}{\delta}\bigg)^{s+\varepsilon}
\]
balls of diameter $\delta/\alpha_2(\textbf{\emph{i}})$, where $D_{1,\varepsilon}$ is a universal constant depending only on $\varepsilon$.  Taking images of these sets under $S_\textbf{\emph{i}}$ gives a cover of $S_\textbf{\emph{i}}(C)$ by ellipses with minor axis $\delta$ and major axis $\delta \alpha_1(\textbf{\emph{i}})/\alpha_2(\textbf{\emph{i}})$.  Projecting each of these ellipses under $\pi_\textbf{\emph{i}}$ gives an interval of length $\delta \alpha_1(\textbf{\emph{i}})/\alpha_2(\textbf{\emph{i}})$, the intersection of which with $\pi_\textbf{\emph{i}}\big(S_\textbf{\emph{i}}(C)\big)$ may be covered by fewer than
\[
D_{2,\varepsilon} \, \bigg(\frac{\delta \alpha_1(\textbf{\emph{i}})/\alpha_2(\textbf{\emph{i}})}{\delta}\bigg)^{\dim_\text{A}\pi_\textbf{\emph{i}}(C)+\varepsilon/2} \ = \ D_{2,\varepsilon} \,\bigg(\frac{\alpha_1(\textbf{\emph{i}})}{\alpha_2(\textbf{\emph{i}})}\bigg)^{\dim_\text{A}\pi_\textbf{\emph{i}}(C)+\varepsilon/2}
\]
intervals of radius $\delta$, where $D_{2,\varepsilon}$ is a universal constant depending only on $\varepsilon$.  Pulling each of these intervals back up to $S_\textbf{\emph{i}}(C)$ and applying assumption (A) gives a $\delta$ cover of $S_\textbf{\emph{i}}(C)$ by fewer than

\begin{eqnarray*}
&\, & D_{1,\varepsilon} \bigg(\frac{\alpha_2(\textbf{\emph{i}})}{\delta}\bigg)^{s+\varepsilon} \ D_{2,\varepsilon} \,\bigg(\frac{\alpha_1(\textbf{\emph{i}})}{\alpha_2(\textbf{\emph{i}})}\bigg)^{\dim_\text{A}\pi_\textbf{\emph{i}}(C)+\varepsilon/2} \\ \\
&\,& \qquad  \leq  D_{1,\varepsilon} \, D_{2,\varepsilon} \,  \delta^{-(s+\varepsilon)} \, \alpha_1(\textbf{\emph{i}})^{\max\{\overline{s}_\textbf{\emph{i}}(C), s_\textbf{\emph{i}}(F_\emptyset) \}} \alpha_2(\textbf{\emph{i}})^{s+\varepsilon/2-\max\{\overline{s}_\textbf{\emph{i}}(C), s_\textbf{\emph{\emph{i}}}(F_\emptyset) \}}
\end{eqnarray*}
which proves the lemma.
\end{proof}

We can now estimate the awkward second term.  We have
\begin{eqnarray*}
\sum_{ \substack{\textbf{\emph{i}} \in \mathcal{I}^* \cup\{ \omega\}: \\ \\
\alpha_2(\textbf{\emph{i}}) > \delta}} N_\delta \big(S_\textbf{\emph{i}}(C) \big) & \leq & D_\varepsilon \,  \delta^{-(s+\varepsilon)}  \sum_{ \substack{\textbf{\emph{i}} \in \mathcal{I}^* \cup\{ \omega\}: \\ \\
\alpha_2(\textbf{\emph{i}}) > \delta}}  \alpha_1(\textbf{\emph{i}})^{\max\{\overline{s}_\textbf{\emph{i}}(C), s_\textbf{\emph{i}}(F_\emptyset) \}} \alpha_2(\textbf{\emph{i}})^{s+\varepsilon/2-\max\{\overline{s}_\textbf{\emph{i}}(C), s_\textbf{\emph{i}}(F_\emptyset) \}}  \\ \\
&\,& \qquad \qquad \qquad \qquad \qquad \qquad \text{by Lemma \ref{assumptionAlemma}} \\ \\
& \leq & D_\varepsilon \,  \delta^{-(s+\varepsilon)}  \sum_{k = 0}^{\infty} \alpha_{\max}^{k\varepsilon/2} \ \Bigg( \sum_{\textbf{\emph{i}} \in \mathcal{I}^k} c_\textbf{\emph{i}}^{\max\{\overline{s}_1(C), s_1(F_\emptyset) \}} d_\textbf{\emph{i}}^{\overline{s}_A - \max\{\overline{s}_1(C), s_1(F_\emptyset) \}}\\ \\
&\,& \qquad \qquad \qquad \qquad \quad \qquad + \ \sum_{\textbf{\emph{i}} \in \mathcal{I}^k} d_\textbf{\emph{i}}^{\max\{\overline{s}_2(C), s_2(F_\emptyset) \}} c_\textbf{\emph{i}}^{\overline{s}_B - \max\{\overline{s}_2(C), s_2(F_\emptyset) \}} \Bigg)\\ \\
& \leq & 2 \, D_\varepsilon \,  \delta^{-(s+\varepsilon)}  \sum_{k = 0}^{\infty} \Big(\alpha_{\max}^{\varepsilon/2}\Big)^k  \qquad \qquad \text{by the definitions of $\overline{s}_A$ and $\overline{s}_B$}\\ \\
& \leq & \frac{2 \, D_\varepsilon}{1- \alpha_{\max}^{\varepsilon/2}} \, \delta^{-(s+\varepsilon)}.
\end{eqnarray*}
Combining the two estimates given above yields
\[
N_\delta(\mathcal{O}) \ \leq \ \Bigg( 2 \, C_\varepsilon \,  \alpha_{\min}^{-2}  \ + \ \frac{2 \, D_\varepsilon}{1- \alpha_{\max}^{\varepsilon/2}} \Bigg) \, \delta^{-(s+\varepsilon)}
\]
which proves that $\overline{\dim}_\text{B} F_C  = \overline{\dim}_\text{B} \mathcal{O} \leq s + \varepsilon$ and letting $\varepsilon$ tend to zero gives the desired upper bound.

\subsection{Proof of the lower bound for the upper box dimension in Theorem \ref{main4}}

In this section we will prove the lower bounds for the the upper box dimension of inhomogeneous self-affine carpets, which, combined with the upper bound in the previous section, yields a precise formula.  We will begin by proving the result in a special case.

\begin{prop} \label{bedmcprop}
Let $F_C$ be in the Lalley-Gatzouras class and assume that $c_i = c \geq d = d_i$ for all $i \in \mathcal{I}$.  Then
\[
\overline{\dim}_\text{\emph{B}} F_C  \geq   \max\{ \overline{s}_A,  \overline{\dim}_\text{\emph{B}} C \}.
\]
\end{prop}

\begin{proof}
Since upper box dimension is monotone, it suffices to show that $\overline{\dim}_\text{B} F_C  \geq  \overline{s}_A$.  Since $\pi_1(\mathcal{O})$ is an inhomogeneous self-similar set which satisfies the OSC, we know from \cite[Corollary 2.2]{me_inhom} that $\overline{\dim}_\text{B} \pi_1(\mathcal{O}) = \max\{\overline{s}_1(C), s_1(F_\emptyset) \}$.  It follows that for all $\varepsilon>0$ we can find infinitely many $k \in \mathbb{N}$ such that
\begin{equation} \label{infiniteupper}
N_{(d/c)^k}\big(\pi_1(\mathcal{O})\big) \geq \big((d/c)^k\big)^{-(\max\{\overline{s}_1(C), s_1(F_\emptyset) \}-\varepsilon)}.
\end{equation}
Fix such a $k$, let $\varepsilon \in (0, \underline{s}_A)$, and $U$ be any closed square of side length $d^k$.  Since $\{S_\textbf{\emph{i}}\big([0,1]^2\big)\}_{\textbf{\emph{i}}\in \mathcal{I}^k}$ is a collection of pairwise disjoint open rectangles each with shortest side having length $d^k$ which is strictly less than the longer side, it is clear that $U$ can intersect no more than $6$ of the sets $\{ S_i(\mathcal{O})\}_{\textbf{\emph{i}}\in \mathcal{I}^k}$ since $S_i(\mathcal{O}) \subseteq S_\textbf{\emph{i}}\big([0,1]^2\big)$ for all $\textbf{\emph{i}}\in \mathcal{I}^k$.  It follows that, using the $\delta$-mesh definition of $N_\delta$, we have
\[
\sum_{\textbf{\emph{i}}\in \mathcal{I}^k} N_{d^k} \big(S_\textbf{\emph{i}}(\mathcal{O})\big) \  \leq  \ 6 \, N_{d^k} \Bigg(\bigcup_{\textbf{\emph{i}}\in \mathcal{I}^k} S_\textbf{\emph{i}}(\mathcal{O}) \Bigg)  \ \leq \  6 \, N_{d^k}(\mathcal{O}).
\]
This yields
\begin{eqnarray*}
N_{d^k}(\mathcal{O})& \geq & \tfrac{1}{6} \  \sum_{\textbf{\emph{i}} \in \mathcal{I}^k} N_{d^k} \big(S_\textbf{\emph{i}}(\mathcal{O})\big) \\ \\
& = & \tfrac{1}{6} \  \sum_{\textbf{\emph{i}} \in \mathcal{I}^k} N_{(d/c)^k} \big(\pi_1(\mathcal{O})\big) \qquad \qquad \text{since $\alpha_2(\textbf{\emph{i}}) = d^k$} \\ \\
&\geq & \tfrac{1}{6} \  \sum_{\textbf{\emph{i}} \in \mathcal{I}^k} \big((d/c)^k\big)^{-(\max\{\overline{s}_1(C), s_1(F_\emptyset) \}-\varepsilon)} \qquad \qquad \text{by (\ref{infiniteupper}) }\\ \\
& \geq & \tfrac{1}{6} \   ({d^k})^{-\overline{s}_A+\varepsilon} \ \Bigg(\sum_{\textbf{\emph{i}} \in \mathcal{I} } c^{\max\{\overline{s}_1(C), s_1(F_\emptyset) \}} d^{\overline{s}_A - \max\{\overline{s}_1(C), s_1(F_\emptyset) \}}\Bigg)^k\\ \\
& \geq & \tfrac{1}{6} \   ({d^k})^{-(\overline{s}_A-\varepsilon)}
\end{eqnarray*}
by the definition of $\overline{s}_A$, which proves that $\overline{\dim}_\text{B} F_C = \overline{\dim}_\text{B}\mathcal{O} \geq   \overline{s}_A - \varepsilon$ and letting $\varepsilon$ tend to zero gives the desired lower bound.
\end{proof}

We will now use Proposition \ref{bedmcprop} to prove the result in the general case.  The key idea is to approximate the IFS `from within' by subsystems which fall into the subclass used in Proposition \ref{bedmcprop}.  This approach is reminiscent of that used by Ferguson, Jordan and Shmerkin when studying projections of carpets \cite[Lemma 4.3]{ferguson_proj}.  There the authors prove that for all $\varepsilon>0$ any Lalley-Gatzouras or Bara\'nski system, $\mathbb{I}$, has a finite subsystem $\mathbb{J}_\varepsilon \subseteq \mathbb{I}^m$ (for some $m \in \mathbb{N}$), with the following properties: $\mathbb{J}_\varepsilon$ consists only of maps with linear part of the form
\[ \left( \begin{array}{cc}
c & 0\\
0 & d \end{array} \right)
\]
for some constants $c,d \in (0,1)$ depending on $\varepsilon$; the Hausdorff dimension of the attractor of $\mathbb{J}_\varepsilon$ is no more than $\varepsilon$ smaller than the Hausdorff dimension of the attractor of $\mathbb{I}$; and $\mathbb{J}_\varepsilon$ has uniform fibers (either vertical or horizontal, depending on the relative size of $c$ and $d$).  It is interesting to note that one cannot approximate the box and packing dimensions `from within' in the same way.  To see this observe that in the uniform fibers case the Hausdorff, box and packing dimensions coincide.  As such if these dimensions did not coincide in the original construction, then one cannot find subsystems for which they coincide but get arbitrarily close to the box dimension.  It is natural to ask if one can do this if the uniform fibers condition is dropped. We have been unable to show this and it seems that the problem is somehow linked to the fact that the packing dimension does not behave well with respect to fixing prescribed frequencies of maps in the IFS.  For examples of such bad behaviour, we note that for Bedford-McMullen carpets there does not usually exist a Bernoulli measure with full packing dimension and the packing spectrum of Bernoulli measures supported on self-affine carpets need not peak at the ambient packing dimension (Thomas Jordan, personal communication).  In contrast to this, there is always a Bernoulli measure with full Hausdorff dimension and the Hausdorff spectrum always peaks at the ambient Hausdorff dimension, \cite{king, jordanrams}.  Also, see the related work of Nielsen \cite{nielsen} on subsets of carpets consisting of points where the digits in the expansions occur with prescribed frequencies. Fortunately, for the purposes of this chapter, we do not need to approximate the box dimension from within, but rather approximate the quantities $\overline{s}_A$ and $\overline{s}_B$, which we \emph{can} do.
\begin{prop} \label{fromwithin}
Let $F_C$ be an inhomogeneous self-affine carpet in the Lalley-Gatzouras or Bara\'nski class and assume that $\overline{s}_1(C) \geq s_1(F_\emptyset)$.  Then for all $\varepsilon>0$, there exists a finite subsystem $\mathbb{J}_\varepsilon = \{S_\textbf{i}\}_{\textbf{i} \in \mathcal{J}_\varepsilon}$ for some $\mathcal{J}_\varepsilon \subseteq \mathcal{I}^m$ and $m \in \mathbb{N}$, with the property that for all $\textbf{i} \in \mathcal{J}_\varepsilon$ we have $c_\textbf{i} = c$, $d_\textbf{i} = d$ for some constants $c,d \in (0,1)$ depending on $\varepsilon$; and the number $\overline{s}_A$ defined by $\mathcal{J}_\varepsilon$ is no more than $\varepsilon$ smaller than the number $\overline{s}_A$ defined by $\mathcal{I}$.
\end{prop}

\begin{proof}
We will use a version of Stirling's approximation for the logarithm of large factorials.  This states that for all $n \in \mathbb{N} \setminus \{1\}$ we have
\begin{equation} \label{stirling}
n \log n - n  \ \leq \ \log n! \ \leq \  n \log n - n +\log n.
\end{equation}
For $i \in \mathcal{I}$, let
\[
p_i = c_i^{\overline{s}_1(C)} d_i^{\overline{s}_A-\overline{s}_1(C)}
\]
and for $k \in \mathbb{N}$, let
\[
m(k) = \sum_{i \in \mathcal{I}} \lfloor p_ik \rfloor \in \mathbb{N}
\]
and note that $k-\lvert \mathcal{I} \rvert \leq m(k) \leq k$.  Consider the $m(k)$th iteration of $\mathcal{I}$ and let
\[
\mathcal{J}_k = \Big\{ \textbf{\emph{j}} = (j_1, \dots, j_{m(k)}) \in \mathcal{I}^{m(k)} : \#\{n : j_n = i\} = \lfloor p_i k \rfloor \Big\}.
\]
It is straightforward to see that %multinomial coefficient
\begin{equation} \label{combina}
\lvert \mathcal{J}_k \rvert = \frac{m(k)!}{\prod_{i \in \mathcal{I}} \lfloor p_ik \rfloor !}
\end{equation}
and for each $\textbf{\emph{j}} \in \mathcal{J}_k$ we have
\[
c_\textbf{\emph{j}} = \prod_{i \in \mathcal{I}} c_i^{\lfloor p_ik \rfloor} =: c
\]
and
\[
d_\textbf{\emph{j}} = \prod_{i \in \mathcal{I}} d_i^{\lfloor p_ik \rfloor} =:  d.
\]
Indeed, these facts were observed in \cite{ferguson_proj}.  We can now use this information to estimate the number $\overline{s}_A$ corresponding to $\mathcal{J}_k$, which we will denote by $\overline{s}_A(\mathcal{J}_k)$ to differentiate it from the number $\overline{s}_A$ corresponding to $\mathcal{I}$, which we will denote by $\overline{s}_A(\mathcal{I})$.  Since $\mathcal{J}_k$ is a subsystem of $\mathcal{I}$ and since $\overline{s}_1(C) \geq s_1(F_\emptyset)$, it follows by definition that
\begin{eqnarray*}
\overline{s}_A(\mathcal{I}) \ \geq \ \overline{s}_A(\mathcal{J}_k) & = & \frac{\log \lvert \mathcal{J}_k \rvert }{-\log d } \ +  \ \overline{s}_1(C) \bigg(1-\frac{\log c}{\log d } \bigg) \\ \\
& = & \frac{\log m(k)! - \sum_{i \in \mathcal{I}} \log  \lfloor p_ik \rfloor!}{-\log d } \ +  \ \overline{s}_1(C) \bigg(1-\frac{\log c}{\log d } \bigg) \qquad \qquad \text{by (\ref{combina})}\\ \\
& \geq &  \frac{m(k)\log m(k) - m(k)  - \sum_{i \in \mathcal{I}}\Big( \lfloor p_ik \rfloor \log  \lfloor p_ik \rfloor - \lfloor p_ik \rfloor + \log \lfloor p_ik \rfloor \Big)}{-\log d } \\ \\
&\,& \qquad  \qquad \quad +  \ \overline{s}_1(C) \bigg(1-\frac{\log c}{\log d } \bigg)  \qquad \qquad \text{by Stirling's approximation (\ref{stirling})}\\ \\
& = &  \frac{m(k)\log m(k) - \sum_{i \in \mathcal{I}} \lfloor p_ik \rfloor \log  \lfloor p_ik \rfloor }{-\log d }  \ +  \ \overline{s}_1(C) \bigg(1-\frac{\log c}{\log d } \bigg) \\ \\
&\,& \qquad  \qquad \qquad \qquad + \  \frac{ \sum_{i \in \mathcal{I}}\log \lfloor p_ik \rfloor }{\log d }  \\ \\
& \geq &  \frac{m(k) \log m(k) - \sum_{i \in \mathcal{I}} \lfloor p_ik \rfloor \log  k c_i^{\overline{s}_1(C)} d_i^{\overline{s}_A(\mathcal{I})-\overline{s}_1(C)} }{-\log d }  \ +  \ \overline{s}_1(C) \bigg(1-\frac{\log c}{\log d } \bigg) \\ \\
&\,& \qquad  \qquad \qquad \qquad + \  \frac{ \sum_{i \in \mathcal{I}}\log \lfloor p_ik \rfloor }{\log d }  \\ \\
& \geq &  \frac{- \sum_{i \in \mathcal{I}} \lfloor p_ik \rfloor \log   c_i^{\overline{s}_1(C)} d_i^{\overline{s}_A(\mathcal{I})-\overline{s}_1(C)} }{-\log d }  \ +  \ \overline{s}_1(C) \bigg(1-\frac{\log c}{\log d } \bigg) \\ \\
&\,& \qquad  \qquad \qquad \qquad + \  \frac{ \sum_{i \in \mathcal{I}}\log \lfloor p_ik \rfloor -m(k)\log(m(k)/k)}{\log d }  \\ \\
& = & \overline{s}_1(C) \frac{- \sum_{i \in \mathcal{I}} \lfloor p_ik \rfloor \log   c_i  }{-\log d }  \ +  \ \Big(\overline{s}_A(\mathcal{I})-\overline{s}_1(C) \Big) \frac{- \sum_{i \in \mathcal{I}} \lfloor p_ik \rfloor \log  d_i}{-\log d }  \\ \\
&\,& \qquad  \qquad \qquad    + \ \overline{s}_1(C) \bigg(1-\frac{\log c}{\log d } \bigg)  \ + \  \frac{ \sum_{i \in \mathcal{I}}\log \lfloor p_ik \rfloor -m(k)\log(m(k)/k)}{\log d }  \\ \\
& =& \overline{s}_1(C) \frac{\log c  }{\log d }  \ +  \ \Big(\overline{s}_A(\mathcal{I})-\overline{s}_1(C) \Big) \ + \ \overline{s}_1(C) \bigg(1-\frac{\log c}{\log d } \bigg)  \\ \\
&\,& \qquad  \qquad \qquad \qquad + \ \frac{ \sum_{i \in \mathcal{I}}\log \lfloor p_ik \rfloor -m(k)\log(m(k)/k)}{\log d }  \\ \\
& = & \overline{s}_A(\mathcal{I}) \  + \ \frac{ \sum_{i \in \mathcal{I}}\log \lfloor p_ik \rfloor -m(k)\log(m(k)/k)}{\log d } \\ \\
&\to&   \overline{s}_A(\mathcal{I})
\end{eqnarray*}
as $k \to \infty$.  It follows that for any $\varepsilon>0$, we can choose $k$ large enough to ensure that the IFS $\mathbb{J}_k = \{S_\textbf{\emph{i}}\}_{\textbf{\emph{i}} \in \mathcal{J}_k}$ satisfies the properties required by $\mathbb{J}_\varepsilon$, which completes the proof.
\end{proof}

We can now complete the proof of the lower bound for the upper box dimension in Theorem \ref{main4}.  We will prove this in the case when $\overline{s}_A \geq \overline{s}_B$.  The other case can clearly be shown by a symmetric argument.

\begin{proof}
We wish to show that $\overline{\dim}_\text{B} F_C  \geq   \max\{ \overline{s}_A,  \overline{\dim}_\text{B} C \}$.  If $\overline{s}_1(C) \leq s_1(F_\emptyset)$, then the result follows by the monotonicity of upper box dimension since in this case $\overline{s}_A \leq \overline{\dim}_\text{B} F$.  If $\overline{s}_1(C) > s_1(F_\emptyset)$, then we may apply Propositions \ref{bedmcprop}-\ref{fromwithin} in the following way.  Let $\varepsilon>0$.  Then by Proposition \ref{fromwithin} there exists a subsystem $\mathcal{J}_\varepsilon$ of the type considered in Proposition \ref{bedmcprop} for which the number $\overline{s}_A = \overline{s}_A(\mathcal{J}_\varepsilon)$ defined by the system $\mathcal{J}_\varepsilon$ is no more than $\varepsilon$ smaller than the number $\overline{s}_A = \overline{s}_A (\mathcal{I})$ defined for the original system $\mathcal{I}$.  Writing $F_C(\mathcal{J}_\varepsilon)$ for the attractor of the IFS corresponding to $\mathcal{J}_\varepsilon$, it follows from Proposition \ref{bedmcprop} that
\[
\overline{\dim}_\text{B} F_C \ \geq \ \overline{\dim}_\text{B} F_C (\mathcal{J}_\varepsilon) \ \geq \ \overline{s}_A(\mathcal{J}_\varepsilon) \ \geq \ \overline{s}_A(\mathcal{I}) - \varepsilon
\]
and letting $\varepsilon$ tend to zero completes the proof.
\end{proof}

\begin{centering}

\textbf{Acknowledgements}

The author was supported by the EPSRC grant EP/J013560/1.  Much of this work was completed whilst the author was an EPSRC funded PhD student at the University of St Andrews and he wishes to express his gratitude for the support he found there.  The author would also like to thank Tuomas Orponen for suggesting the construction of $C$ in Section \ref{example2} and Thomas Jordan for useful discussions.

\end{centering}

\end{document}